\documentclass[10pt]{article}
\usepackage[all]{xy}
\usepackage{amsfonts,amsmath,oldgerm,amssymb,amscd}
\newcommand{\ra}{\rightarrow}		
\newcommand{\by}[1]{\stackrel{#1}{\ra}}

\newcommand{\surj}{\ra\!\!\!\ra}	
\newcommand{\ol}{\overline}		\newcommand{\wt}{\widetilde}
\newcommand{\iso}{\by \sim}

\newtheorem{theorem}{Theorem}[section]
\newtheorem{proposition}[theorem]{Proposition}
\newtheorem{lemma}[theorem]{Lemma}

\newtheorem{corollary}[theorem]{Corollary}
\newtheorem{question}[theorem]{Question}

\newcommand{\ga}{\alpha}	
		\newcommand{\gd}{\delta}
	
\newcommand{\gj}{\blacksquare}

\newcommand{\gt}{\theta}

	\newcommand{\gD}{\Delta}

	\newcommand{\BN}{\mbox{$\mathbb N$}}

	\newcommand{\CP}{\mbox{$\mathcal P$}}

\newcommand{\ma}{\mbox{$\mathfrak a$}}

\newcommand{\mm}{\mbox{$\mathfrak m$}}	
	\newcommand{\pri}{\mbox{$\mathfrak p$}}

\newcommand{\ot}{\mbox{\,$\otimes$\,}}	\newcommand{\op}{\mbox{$\oplus$}}
\newcommand{\Spec}{\mbox{\rm Spec\,}}	\newcommand{\hh}{\mbox{\rm ht\,}}

\newcommand{\Um}{\mbox{\rm Um}}

\begin{document}

\begin{center}
 {\Large \bf Serre dimension and Euler class groups of  overrings  of  polynomial rings}\\\vspace{.2in}
  {\large Manoj K. Keshari and Husney Parvez Sarwar }\\
\vspace{.1in}
{\small
Department of Mathematics, IIT Bombay, Powai, Mumbai - 400076, India;\;\\
    (keshari, parvez)@math.iitb.ac.in}
\end{center}

\begin{abstract}
 
Let $R$ be a commutative Noetherian ring of dimension $d$ and
$B=R[X_1,\ldots,X_m,Y_1^{\pm 1},\ldots,Y_n^{\pm 1}]$ a Laurent
polynomial ring over $R$. If $A=B[Y,f^{-1}]$ for some $f\in R[Y]$,
then we prove the following results:

$(i)$ If $f$ is a monic polynomial, then Serre dimension of $A$ is
$\leq d$. In case $n=0$, this result is due to Bhatwadekar, without
the condition that $f$ is a monic polynomial.

$(ii)$ The $p$-th Euler class group $E^p(A)$ of $A$, defined by
Bhatwadekar and Raja Sridharan, is trivial for $p\geq max \{d+1,
\dim A -p+3\}$. In case $m=n=0$, this result is due to Mandal-Parker.
\end{abstract}

\section{Introduction}
{\it In this paper, we will assume that all rings are commutative
  Noetherian of finite Krull dimension, all modules are finitely
  generated and all projective modules are of constant rank. } {\bf
  Throughout this paper, $R$ will denote a ring of dimension $d$ and
  $B$ will denote the Laurent polynomial ring $R[X_1,\ldots,X_m,Y_1^{\pm
      1},\ldots, Y_n^{\pm 1}]$ over $R$.}

Let $P$ be a projective $R$-module. An element $p\in P$
is said to be {\it unimodular} if there exist $\phi\in Hom(P,R)$ such
that $\phi(p)=1$. We write $\Um (P)$ for the set of all unimodular
elements of $P$. We say that {\it Serre dimension} of $R$ is $\leq t$ if every
projective $R$-module of rank $\geq t+1$ has a unimodular element.  

It is a classical result due to Serre \cite{Serre58} that Serre
dimension of $R$ is $\leq d$. Let $R[Y]$ be a polynomial ring in one
variable over $R$, $P$ a projective $R[Y]$-module and $f$ any monic
polynomial in $R[Y]$. If we further assume that $R$ is a local ring,
then Horrocks (\cite{H}, Theorem 1) proved that projective
$R[Y,f^{-1}]$-module $P_f$ is free implies $P$ itself is free.
Quillen extended Horrocks' theorem to arbitrary ring $R$ (\cite{Qu},
Theorem 3). Using this result, Quillen (\cite{Qu}, Theorem 4) proved
Serre's conjecture that projective modules over polynomial rings
$k[X_1,\ldots,X_m]$ over a field $k$ are free for all $m\geq 1$. In
other words, Serre dimension of $k[X_1,\ldots,X_m]$ is $0$.

Plumstead generalized Serre's result (\cite{P}, Theorem 2) by proving
that Serre dimension of $R[Y]$ is $\leq d$. Rao generalized
Plumstead's result (\cite{Rao82}, Theorem 1.1) and proved that if $C$
is a birational overring of $R[Y]$, i.e. $R[Y]\subset C \subset
qf(R[Y])=S^{-1}R[Y]$, where $S$ is the set of all non-zerodivisors of
$R[Y]$, then Serre dimension of $C$ is $\leq d$. As a consequence of
Rao's result, we get that Serre dimension of $R[Y,f^{-1}] \leq d$ for
any non-zerodivisor $f\in R[Y]$.

Bhatwadekar-Roy (\cite{BhR}, Theorem 3.1) generalized Plumstead's
result to polynomial rings in many variables and proved that Serre
dimension of polynomial ring $R[X_1,\ldots,X_m]$ is $\leq d$ for any
$m\geq 1$. This result of Bhatwadekar-Roy was generalized by
Bhatwadekar-Lindel-Rao (\cite{BLR}, Theorem 4.1) to Laurent polynomial
case. They proved that Serre dimension of Laurent polynomial ring
$B:=R[X_1,\ldots, X_m,Y_1^{\pm 1},\ldots,Y_n^{\pm 1}]$ is $\leq d$.

Bhatwadekar (\cite{B1}, Theorem 3.5) further generalized above result of
Bhatwadekar-Roy to polynomial extensions over a birational overring of
$R[Y]$. More precisely, he proved that if $C$ is a birational overring
of $R[Y]$, then Serre dimension of $C[X_1,\ldots,X_m]$ is $\leq d$. As
a consequence of this result, we get that Serre dimension of
$R[X_1,\ldots,X_m,Y,f^{-1}]$ is $\leq d$ for any non-zerodivisor $f\in
R[Y]$. 

It is natural to ask if analogue of Bhatwadekar's result \cite{B1} is true for
Laurent polynomial rings. More precisely, we can ask the following.

\begin{question}
Let $C$ be a birational
overring of $R[Y]$. Is Serre dimension of
$C[X_1,\ldots,X_m,Y_1^{\pm 1},\ldots,Y_n^{\pm 1}]$ $\leq d$?
\end{question} 

We answer this question when $C=R[Y,f^{-1}]$ with $f\in R[Y]$ a monic
polynomial. Note that Lindel \cite{L95} gave another proof of
Bhatwadekar-Lindel-Rao's result (\cite{BLR}, Theorem 4.1) mentioned
above. Our proof closely follows Lindel's idea. We state our result.
 
 \begin{theorem}\label{T1}
Let $A=B[Y,f^{-1}]$, where $f\in R[Y]$ is a monic polynomial. Then
Serre dimiension of $A$ is $\leq d$.
 \end{theorem}

Assume $\dim R=d\geq 3$ and $p$ is a positive integer such that $p\geq
d-p+3$. Then Bhatwadekar and Raja Sridharan have defined the $p$-th
Euler class group $E^p(R)$ of $R$ which is an additive abelian
group. We will not give the explicit definition of $E^p(R)$ (see
\cite{BR1}, section 4 for definition). Rather we will describe the
elements of $E^p(R)$, since this suffices for our purpose. Let $I$ be
an ideal of $R$ of height $p$ such that the $R/I$-module $I/I^2$ is
generated by $p$ elements. Let $\phi:(R/I)^p \surj I/I^2$ be a
surjection, giving a set of $p$ generators of $R/I$-module
$I/I^2$. The surjection $\phi$ induces an element of the $p$-th Euler
class group $E^p(R)$, denote it by the pair $(I,\phi)$. Further, it
follows using {\it moving lemma and addition principle}, that every
element of $E^p(R)$ is a pair $(I,\phi)$ for some height $p$ ideal $I$
of $R$ and some surjection $\phi : (R/I)^p \surj I/I^2$.  Bhatwadekar
and Raja Sridharan (\cite{BR1}, Theorem 4.2) proved that there exist a
surjection $\Phi:R^p \surj I$ which is a lift of $\phi$, i.e. $\Phi\ot
A/I =\phi$, if and only if the associated element $(I,\phi)$ of the
group $E^p(R)$ is the trivial element (identity element $0$ of
$E^p(R)$).

It is well known that a projective $R$-module of rank $d$ need
not, in general, have a unimodular element. The significance of Euler
class group theory is demonstrated by the following result of
Bhatwadekar-Raja Sridharan \cite{BR2}, where they proved that for a
rank $d$ projective $R$-module $P$ with trivial determinant, 
the precise obstruction for $P$ to have a unimodular element lies
in $E^d(R)$. More precisely, given a pair $(P,\chi)$, where $\chi:
\wedge^d P \iso R$ is an isomorphism, they associate an element
$e(P,\chi)$ of the Euler class group $E^d(R)$ and prove that $P$ has a
unimodular element if and only if $e(P,\chi)$ is the trivial element
of $E^d(R)$. Such an obstruction theory is not known for projective
$R$-modules of rank $d-1$ except for some special class of rings.
When $R=S[Y]$ is a polynomial ring in one variable over some
subring $S$ of $R$, then Das \cite{D1} proved that for a rank $d-1$
projective $R$-module $Q$ with trivial determinant, the precise
obstruction for $Q$ to have a unimodular element lies in $E^{d-1}(R)$.

Let $I$ be an ideal of $R[Y]$ containing a monic polynomial in the
variable $Y$. Assume $R[Y]/I$-module $I/I^2$ is generated by $p$
elements, where $p\geq \dim (R[Y]/I) +2$. Then Mandal (\cite{Man},
Theorem 2.1) proved that any surjection $\phi:(R[Y]/I)^p \surj I/I^2$
can be lifted to a surjection $\Phi:R[Y]^p \surj I$. Let $P=Q\op R$ be
a projective $R$-module of rank $p$ and $\psi:P[Y]/IP[Y] \surj I/I^2$
be a surjection, then Bhatwadekar-Raja Sridharan (\cite{BR2},
Proposition 3.3) proved that $\psi$ lifts to a surjection $\Psi:P[Y]
\surj I$, thus generalizing Mandal's result.
If we further
assume that height of $I$ is $p$ and $2p \geq \dim R[Y]+3$, then to 
the surjection $\phi$, we can associate an element $(I,\phi)\in
E^p(R[Y])$. Since $\Phi$ is a surjective lift of $\phi$, by
(\cite{BR1}, Theorem 4.2), we get that $(I,\phi)$ is a trivial element of
$E^p(R[Y])$. 

Let $A=R[X_1,\ldots,X_m]$ be a polynomial ring over $R$ and $I$ an
ideal of $A$ of height $\geq d+1$. Let $p\geq max\{\dim (A/I)+2,d+1\}$
be an integer and $\phi:(A/I)^p \surj I/I^2$ be a surjection. Since height
of $I>d$, by Suslin (\ref{4t1}), there exist an automorphism $\Theta$
of $A$ such that $\Theta(I)$ contains a monic polynomial in $X_m$ with
coefficients from $R[X_1,\ldots,X_{m-1}]$. Therefore replacing $I$ by
$\Theta(I)$, we may assume that $I$ contains a monic polynomial in
$X_m$. By Mandal (\cite{Man}, Theorem 2.1) mentioned above, $\phi$ can
be lifted to a surjection $\Phi:A^p \surj I$. Therefore if we further
assume that $p\geq max\{\dim A-p+3,d+1\}$, then by (\cite{BR1},
Theorem 4.2), the associated element $(I,\phi)$ of $E^p(A)$ is
trivial. Since any element of $E^p(A)$ is a pair $(I,\phi)$ for some
height $p$ ideal $I$ of $A$, we get that the $p$-th Euler class group
$E^p(A)=0$.  In particular, $E^{d+1}(R[Y])=0$ for $d\geq 2$. This
result is generalized by Mandal-Parker (\cite{M-Pa}, Theorem 3.1) where
they prove that $E^{d+1}(R[Y,f^{-1}])=0$ for $d\geq 2$ and $f\in
R[Y]$. We generalize Mandal-Parker's result as follows.

\begin{theorem}\label{T2}
Let $A=B[Y,f^{-1}]$ for some $f\in R[Y]$ and $p$ an integer such that
$p\geq max \{\dim A-p+2,d+1\}$. Let $P=Q\op R$ be a projective
$R$-module of rank $p$ and $I$ a proper ideal of $A$ of height $\geq
d+1$. Assume there is a surjection $\phi : P\ot A/I(P\ot A)\surj
I/I^2$. Then $\phi$ can be lifted to a surjection $\Phi:P\ot A \surj
I$.  As a consequence, taking $P$ to be free, we get that any $p$
generators of $I/I^2$ can be lifted to $p$ generators of $I$.
\end{theorem}

The following result is a direct consequence of $(\ref{T2})$.

\begin{corollary}
Let $A=B[Y,f^{-1}]$ for some $f\in R[Y]$ and $p$ an integer such that
$p\geq max\{\dim A-p+3,d+1\}$. Then the $p$-th Euler class group
$E^p(A)$ of $A$ is zero.
\end{corollary}

Let $I$ be an ideal of $R[Y]$ containing a monic polynomial and $P$ a
projective $R$-module of rank $p$ with $p\geq \dim (R[Y]/I)+2$. Let
$\phi: P[Y]/IP[Y] \surj I/I^2$ and $\gd:P\surj I(0):=\{f(0)|f\in I\}$
be two surjections such that $\phi(0) =\gd \ot R/I(0)$.
Then Mandal (\cite{Man}, Theorem 2.1)
proved that there exists a surjection $\Phi:P[Y]\surj I$ such that
$\Phi \ot R[Y]/I=\phi$ and $\Phi(0)=\gd$, thus answering a question of
Nori (see \cite{Man}) on Homotopy sections of projective modules, in
case the ideal $I$ contains a monic polynomial.

Above result of Mandal on homotopy section was generalised by
Datt-Mandal (\cite{MD}, Theorem 1.2) to Laurent polynomial case as
follows: Let $I$ be an ideal of $R[Y,Y^{-1}]$ containing a monic
polynomial $f$ in $R[Y]$ with $f(0)=1$. Let $P$ be a projective
$R$-module of rank $p$ with $p\geq \dim (R[Y,Y^{-1}]/I)+2$. Let $\phi:
P[Y,Y^{-1}]/IP[Y,Y^{-1}] \surj I/I^2$ and $\gd:P\surj
I(1):=\{g(Y=1)|g\in I\}$ be two surjections such that $\phi(1) =\gd
\ot R/I(1)$. Then there exists a surjection $\Phi:P[Y,Y^{-1}]\surj I$
such that $\Phi \ot R[Y,Y^{-1}]/I=\phi$ and $\Phi(1)=\gd$.

We prove the following result which is an analogue of Datt-Mandal's
result. 

\begin{theorem}\label{T3}
 Let $A=B[Y,f^{-1}]$, where $f\in R[Y]$ is a monic polynomial with
 $f(1)$ a unit in $R$.  Let $I$ be an ideal of $A$ and $P$ a
 projective $B$-module of rank $\geq max\{d+1, \dim (A/I)+2)\}$.  Let
 $\phi: P[Y,f^{-1}]/IP[Y,f^{-1}]\surj I/I^2$ and $\gd: P\surj I(1)$ be
 two surjections such that $\phi = \gd \ot A/(Y-1)$, where $I(1)$ is
 an ideal of $B$.  Then there exist a surjection $\Psi :
 P[Y,f^{-1}]\surj I$ such that $\Psi \ot A/I=\phi$ and $\Psi(1)=\gd$.
\end{theorem}

 \section{Preliminaries}
In this section, we note down some results for later use.
For a ring $A$, $\hh I$ will denote the height of an ideal $I$ of $A$.
We begin by stating a result of Lindel (\cite{L95}, Lemma 1.1).
 
\begin{proposition}\label{4p1}
Let $A$ be a ring, $Q$ an $A$-module and $s\in A$ such that $Q_s$ is free
$A_s$-module of rank $r$. Then there exist $p_1,\ldots,p_r \in Q$,
$\phi_1,\ldots,\phi_r\in Q^*$ and $t\geq 1$ such that

$(i)$ $0:_A s'A=0:_A{s'}^{2}A$, where $s'=s^t$.
 
$(ii)$ $s'Q \subset F$ and $s'Q^*\subset G$, where $F=\sum_{i=1}^r
 Ap_i\subset Q$ and $G=\sum _{i=1}^r A\phi_i \subset Q^*$.
  
$(iii)$ $(\phi_i(p_j))_{1\leq i,j\leq r}=diagonal\,(s',\ldots,s')$. We
 say $F$ and $G$ are $s'$-dual submodules of $Q$ and $Q^*$
 respectively.
 \end{proposition}

The following result on fiber product is well known. For a reference
(see \cite{Man2}, Proposition 2.2.1).

\begin{proposition}\label{4p3}
 Let $A$ be a ring and $f,g\in A$ be such that $fA+gA=A$. Let $M$ and
 $N$ be two $A$-modules. Suppose $\phi: M_f\ra N_f$ is an
 $A_f$-homomorphism and $\psi:M_g\ra N_g$ is an $A_g$-homomorphism
 such that $\phi_g=\psi_f$.  Then
 
 $(i)$ there exist an $A$-homomorphism $\xi: M\ra N$ such that 
 $\xi_f=\phi$ and $\xi_g=\psi$.
 
 $(ii)$ if $\phi$ and $\psi$ are surjective, then $\xi$ is surjective.
\end{proposition}

The following is implicit in Suslin's result (\cite{Su}, Lemma 6.2)
and is known as Suslin's monic polynomial theorem.

\begin{theorem}\label{4t1}
Let $I$ be an ideal of $R[X_1,\ldots,X_m]$ of height $> d$. Then there
exist a positive integer $N$ such that for any integers $s_i>N$ if
$\phi$ is the $R[X_m]$-automorphism of $R[X_1,\ldots,X_m]$ defined by
$\phi(X_i)= X_i+X_m^{s_i}$ for $1\leq i\leq m-1$, then $\phi(I)$
contains a monic polynomial in $X_m$ with coefficients from
$R[X_1,\ldots,X_{m-1}]$.
 
\end{theorem}

The following result is implicit in Mandal's result (\cite{Man1}, Lemma 2.3).

\begin{lemma}\label{4l2}
 Let $I$ be an ideal of $B$ of height $>d$ and $n>0$. Then there
 exist a $R[Y_n^{\pm 1}]$-automorphism $\Theta$ of $B$ such that
 $\Theta(I)$ contains a monic polynomial in $Y_n$ of the form $1+Y_nh$
 for some $h\in R[X_1,\ldots,X_m,Y_1^{\pm 1},\ldots,Y_{n-1}^{\pm
     1},Y_n]$.
\end{lemma}

The following result is due to Bhatwadekar-Lindel-Rao (\cite{BLR},
Theorem 4.1).

\begin{theorem}\label{4t3}
Let $P$ be a projective $B$-module of rank $>d$. Then $P$ has a
unimodular element.
\end{theorem}

The following result is due to Bhatwadekar-Raja Sridharan (\cite{BR3},
Proposition 3.3).

\begin{proposition}\label{4p4}
 Let $I$ be an ideal of $R[X]$ containing a monic polynomial and
 $P=Q\op A$ a projective $R$-module of rank $r$, where $r\geq \dim
 (R[X]/I)+2$.  Let $\phi: P[X]\surj I/I^2$ be a surjection. Then
 $\phi$ can be lifted to a surjection $\Phi:P[X]\surj I$.
\end{proposition}

The following result is due to Dhorajia-Keshari (\cite{DK1}, Theorem
3.12). We will only state the part needed here.

\begin{theorem}\label{DK2}
Let $A=R[X_1,\ldots,X_m,Y_1,\ldots,Y_n,(f_1\ldots f_n)^{-1}]$ with
$f_i\in R[Y_i]$ and $P$ a projective $A$-module of rank $r\geq
d+1$. Then $P$ is cancellative, i.e. $P\op A^t \iso Q\op A^t$
for some integer $t>0$ implies $P\iso Q$.
\end{theorem}

\begin{define}
For an integer $n>0$, a sequence of elements
 $a_1,\ldots,a_n$ in $R$ is said to be a {\it regular sequence} of
 length $n$ if $a_i$ is a non-zerodivisor in $R/(a_1,\ldots,a_{i-1})$
 for $i=1,\ldots,n$.

Let $I$ be an ideal of $R$.  We say $I$ is {\it set theoretically}
generated by $n$ elements $f_1,\ldots,f_n \in R $ if
$\sqrt{I}=\sqrt{(f_1,\ldots,f_n)}$.

Assume height of $I$ is $n$. Then $I$ is said to be a
 {\it complete intersection} ideal if $I$ is generated by a regular
 sequence of length $n$. Further,
$I$ is said to be a {\it locally complete intersection}
ideal if $I_{\pri}$ is a complete intersection ideal of
height $n$ for all prime ideals $\pri$ of $R$ containing $I$.
$\hfill \gj$
\end{define}
\medskip
 
 The following result is due to Mandal-Roy (\cite{M-R}, Theorem
 2.1). See also (\cite{Man1} Theorem 6.2.2).

 \begin{theorem}\label{4t4}
  Let $J\subset I$ be two ideals of $R[X]$ such that
  $I$ contains a monic polynomial. Assume $I=(f_1,\ldots,f_n)+ I^2$
  and $J=(f_1,\ldots,f_{n-1}) + I^{(n-1)!}$.  Then $J$ is generated by
  $n$ elements. As a consequence, since $\sqrt I=\sqrt J$, $I$ is
  set-theoretically generated by $n$ elements.
 \end{theorem}

The following result is due to Ferrand and Szpiro. For a proof see
\cite{Sz} or \cite{Mu1}.

\begin{theorem}\label{4t6}
Let $I$ be a locally complete intersection ideal of $R$ of height
$n\geq 2$ with $\dim (R/I)\leq 1$.  Then there is a locally complete
intersection ideal $J\subset R$ of height $n$ such that 

$(i)$ $\sqrt{I}=\sqrt{J}$ and

$(ii)$ $J/J^2$ is a free $R/J$-module of rank $n$.
\end{theorem}

The following result is easy to prove, hence we omit the proof.

\begin{lemma}\label{4l1}
Let $f\in R[T]-R$. Then

$(i)$ If $I$ is a proper ideal of $
 R[T,f^{-1}]$, then  $\hh I= \hh (I\cap R[T])$.

$(ii)$ If $I$ is a proper ideal of $
 R[f,f^{-1}]$, then  $\hh I= \hh (I\cap R[f^{-1}])$.
\end{lemma}

\begin{lemma}\label{4l3}
Let $I$ be an ideal of $A=R[T,f^{-1}]$, where $f\in R[T]-R$.  If
$J=I\cap R[f^{-1}]$, then $\hh J=\hh I$.
\end{lemma}

\begin{proof}
Assume that $I$ is a prime ideal. If we write $\ma =I \cap R$, then
$\hh I=\hh IA_{\ma}$ and $\hh J=\hh JR_{\ma}[f^{-1}]$.  Hence we
assume that $(R,\ma)$ is a local ring. Further if $I=\ma A$ is an
extended ideal, then $\hh I=\hh \ma=\hh J$. Hence assume that $I\neq
\ma A$. In this case $\hh I=\hh \ma +1$. Since $R/\ma$ is a field, we
get that $R/\ma[f,f^{-1}]\ra R/\ma[T,f^{-1}]$ is an integral
extension. Hence $\hh I/\ma=\hh \wt J/\ma$, where $\wt J=I\cap
R[f,f^{-1}]$. Therefore $\hh I=\hh \ma+1 =\hh \wt J=\hh J$, by
(\ref{4l1}). The general case follows by noting that $\hh I= \hh \sqrt
I$, $\sqrt I=\CP_1\cap\ldots \cap\CP_r$, $\sqrt J=\CP'_1\cap\ldots
\cap\CP'_r$, where $\CP_i'=\CP_i\cap R[f^{-1}]$
and $\hh \CP_i=\hh \CP_i'$.
$\hfill \gj$
\end{proof}

\begin{proposition}\label{4p2}
 Let $A=B[Y,f^{-1}]$, where $f\in R[Y]$ is a monic polynomial and $I$
 an ideal of $A$ of height $> d$. Then there exist an integer $N>0$
 such that for any set of integers $t_i,s_i,l_i$ all bigger than $N$,
 the $R[Y,f^{-1}]$-automorphism $\phi$ of $A$ defined by $\phi(X_i)=
 X_i+Y^{t_i}+f^{-s_i}$ and $\phi(Y_i)= Y_if^{l_i}$ satisfies the
 following:

$(i)$ $\phi(I)$ contains a monic polynomial in $Y$ with coefficients
 from $B$ and

$(ii)$ $\phi(I)$ contains a polynomial of the form $1+fh$ for some
 $h\in B[Y]$.
\end{proposition}

\begin{proof}
If $n=0$, then $B=R[X_1,\ldots,X_m]$.  If $I_1=I\cap B[f^{-1}]$, then
by (\ref{4l3}), $\hh I_1 = \hh I>d$. Hence by (\ref{4t1}), we can find
a positive integer $N_1$ such that for any integers $s_i>N_1$, if
$\phi_1$ is the $R[f^{-1}]$-automorphism of $B[f^{-1}]$
defined by $\phi_1(X_i)=X_i+f^{-s_i}$ for $1\leq i\leq m$, then
$\phi_1(I_1)$ contains a monic polynomial, say $F$ of degree $u$, in
the variable $f^{-1}$ with coefficients from
$B$. Since $\phi_1$ naturally extends to an
$R[Y,f^{-1}]$-automorphism of $A$, we get $\phi_1(I)$ contains $F$ and
hence it contains $f^uF$ which is of the form $1+fg$ for some $g\in
B[Y]$.

If $I_2=\phi_1(I)\cap B[Y]$, then by (\ref{4l1}), $\hh
I_2 = \hh I> d$. Hence by (\ref{4t1}), we can find a positive integer $N_2$
such that for any integers $t_i>N_2$, if $\phi_2$ is the
$R[Y]$-automorphism of $B[Y]$ defined by $\phi_2(X_i)=
X_i+Y^{t_i}$ for $1\leq i \leq m$, then $\phi_2(I_2)$ contains a monic
polynomial, say $G$, in the variable $Y$ with coefficients from
$B$. Since $\phi_2$ naturally extends to an
$R[Y,f^{-1}]$-automorphism of $A$, we get that
$\phi_2\phi_1(I)$ contains

$(i)$ a monic polynomial $G$ in the
variable $Y$ with coefficients from $B$ and 

$(ii)$ an element $1+fh$, where $h=\phi_2(g)\in B[Y]$.

Note that $\phi_2\phi_1$ is an $R[Y,f^{-1}]$-automorphism of $A$ defined
by $X_i\mapsto X_i+Y^{t_i}+f^{-s_i}$. This proves the result in case
$n=0$ by taking $N=max\{N_1,N_2\}$.

Assume $n>0$ and use induction on $n$. Define $L_{Y_n}(I)$ and
$L_{{Y_n}^{-1}}(I)$ as the set of highest degree coefficients and
lowest degree coefficients respectively of elements in $I$ as a
Laurent polynomial in the variable $Y_n$. It is easy to see that
$L_{Y_n}(I)$ and $L_{{Y_n}^{-1}}(I)$ are ideals of $C[Y,f^{-1}]$,
where $C=R[X_1,\ldots,X_m,Y_1^{\pm 1},\ldots,Y_{n-1}^{\pm 1}]$. By
(\cite{Man1}, Lemma 3.1), we get that height of the ideals
$L_{Y_n}(I)$ and $L_{{Y_n}^{-1}}(I)$ are $\geq \hh I$.

If we write $L=L_{Y_n}(I) \cap L_{{Y_n}^{-1}}(I)$, then $L$ is an
ideal of $C[Y,f^{-1}]$ of height $\geq \hh I>d$. Hence by
induction on $n$, there exist an integer $N_3$ such that for any set
of integers $t_i,s_i,l_i$ all bigger that $N_3$, if $\gt_1$ is an
$R[Y,f^{-1}]$-automorphism of $C[Y,f^{-1}]$ defined by $\gt_1(X_i)=
X_i+Y^{t_i}+f^{-s_i}$ and $\gt_1(Y_j)=Y_jf^{l_j}$ for $1\leq i \leq m$
and $1\leq j \leq n-1$, then $\gt_1(L)$ contains

$(a)$ a monic polynomial, say $\wt g$, in $Y$ with coefficients from
$C$ and 

$(b)$ a polynomial $\wt h$ of the form $1+fh'$ for some $h'\in
C[Y]$. 

We extend $\gt_1$ to an $R[Y_n^{\pm 1},Y,f^{-1}]$-automorphism of $A$.
We can find polynomials $F$ and $G$ in $\gt_1(I)$ of the form $F= \wt
gY_n^s+g_{n-1}Y_n^{s-1}+\cdots +g_0$ and $G=\wt
h+h_1Y_n+\cdots+h_tY_n^{t}$ for some $s,t\in \BN$, $g_i,h_i\in
C[Y,f^{-1}]$ and $\wt g,\wt h$ as in $(a)$, $(b)$. 
We can choose an integer $N_4>0$ such that for any integer
$l_n>N_4$, if $\gt_2$ is an $C[Y,f^{-1}]$-automorphism of
$A$ defined by $\gt_2(Y_n)= Y_nf^{l_n}$, then

$(i)$ $\gt_2(Y_n^{-s}F)$ is a monic polynomial in $Y$ with
coefficients from $C[Y_n^{\pm 1}]=B$ and

$(ii)$ $\theta_2(G)=1+fh$ for some $h\in B[Y]$.

We note that $\gt_2\gt_1$ is an $R[Y,f^{-1}]$-automorphism of $A$
defined by $X_i\mapsto X_i+Y^{t_i}+f^{-s_i}$ and $Y_j\mapsto
Y_jf^{l_j}$ for $1\leq i\leq m$ and $1\leq j\leq n$. Taking
$N=max\{N_3,N_4\}$ completes the proof.  $\hfill \gj$
\end{proof}


\section{Main Theorems}
In this section, we prove results stated in the introduction.

\subsection{Proof of Theorem \ref{T1}}
Without loss of generality, we may assume that $R$ is reduced.  If
$m=0$, then replacing $A$ by $A[X_1]$, we will assume that $m>0$.  Let
$P$ be a projective $A$-module of rank $r>d=\dim R$. We need to show
that $P$ has a unimodular element.  If $S$ denote the set of all
non-zerodivisors of $R$, then $S^{-1}R$ is a zero dimensional
ring. Therefore, by Dhorajia-Keshari (\cite{DK1}, Lemma 3.9),
$S^{-1}P$ is a stably free $S^{-1}A$-module. Hence by (\ref{DK2}),
$S^{-1}P$ is a free $S^{-1}A$-module.  Since $P$ is finitely
generated, we can find some $s\in S$ such that $P_s$ is a free
$A_s$-module of rank $r$.  By Lindel (\ref{4p1}), there exist an
integer $t>0$, $p_1,\ldots,p_r\in P$ and $\phi_1,\ldots,\phi_r \in
P^*$ such that the submodules $F= \sum_{i=1}^rAp_i$ of $P$ and $G=
\sum_{i=1}^{r}A\phi_i$ of $P^*$ satisfies the followings: $s^tP\subset
F$, $s^tP^*\subset G$ and the matrix
$(\phi_i(p_j))=diag(s^t,\ldots,s^t)$. The submodules $F$ and $G$ are
called $s^t$-dual submodules of $P$ and $P^*$ respectively.  Replacing
$s$ by $s^t$, we assume that $F$ and $G$ satisfies $sP\subset F$,
$sP^*\subset G$ and $(\phi_i(p_j))=diag(s,\ldots,s)$.

Since $A/(s(Y-1))=\wt R[X_1,\ldots,X_m,Y_1^{\pm 1},\ldots,Y_n^{\pm
    1}]$ is a Laurent polynomial ring over a $d$ dimensional ring $\wt
R:=R[Y,f^{-1}]/(s(Y-1))$, by Bhatwadekar-Lindel-Rao (\ref{4t3}),
$P/(s(Y-1))$ has a unimodular element.  Let $p\in P$ be such that its
image $\ol p$ in $P/s(Y-1)P$ is a unimodular element.

Let us write $\phi_i(p)=a_i\in A$ for $1\leq i \leq r$ and define
$b:=(1-Y)\prod_{i=1}^{m}X_i\prod_{j=1}^{n}Y_j$. Then $sb$ is a
non-zerodivisor in $A$. We can find an integer $l>deg(a_1)$ such that
$a_1':=a_1+s^2b^l$ is a non-zerodivisor in $A$, where $deg(a_1)$ is the
total degree of $a_1$ as a polynomial in $X_1,\ldots,X_m$ with
coefficients from $R[Y_1^{\pm 1},\ldots,Y_n^{\pm 1},Y,f^{-1}]$. Hence
height of the ideal $a_1'A$ is $\geq 1$.

Since $\ol p$ is a unimodular element in $P/s(Y-1)P$ and
$\phi_1,\ldots,\phi_r$ is a basis of the free module $P_s^*$, we get
that $(a_1,a_2,\ldots,a_r,s^2(Y-1)) \in \Um_{r+1}(A_s)$.  Since $a'\in
a_1+s^2(Y-1)A$, we get $(a_1',a_2,\ldots,a_r,s^2(Y-1))\in
\Um_{r+1}(A_{s})$. Hence by prime avoidance argument, we can choose
$c_2,\ldots,c_{r}$ in $A$ such that if $a_i'= a_i+s^2(Y-1)c_i$ for
$2\leq i \leq r$, then height of the ideal
$(a_1',\ldots,a_r')A_{s(Y-1)}$ is $\geq r $.  Let $l'>2\wt d$ be an
integer, where $\wt d$ is the maximum of total degrees of
$a_1',\ldots,a_r'$ as a polynomial in $X_1,\ldots,X_m$.
If we
write $a_r'':=a_r'+s^2(Y-1)(a_1')^{l'}$, then degree of $a_r''$ as a
polynomial in $X_1,\ldots,X_m$ is $e':=mll'$.

Let $q=c_2p_2+\cdots+c_{r-1}p_{r-1}+(c_r+(a_1')^{l'})p_r$. Then $\wt p:=p+
sb^lp_1+s(Y-1)q$ is also a lift of $\ol p$. Further we have
$\phi_i(\wt p)=a_i'$ for $1\leq i\leq r-1$ and $\phi_r(\wt
p)=a_r''$. Hence replacing $p$ by $\wt p$, we see that height of the
ideal $O_P(p)A_{s(Y-1)}=(a_1',\ldots,a_{r-1}',a_r'')A_{s(Y-1)}$ is $\geq r$.
 
Since $\ol p$ is a unimodular element in $P/s(Y-1)P$ and $p\in P$ is a
lift of $\ol p$, we get $O_P(p)+s(Y-1)A=A$. Further height of the
ideal $O_P(p)A_{s(Y-1)}$ is $\geq r$. Therefore we get that height of
the ideal $O_P(p)$ is $\geq r$.  By (\ref{4p2}), there exist an
integer $N>0$ such that for any integers $t',s',l''$ all bigger than
$N$, if $\Theta$ is the $R[Y,f^{-1}]$-automorphism of $A$ defined by
$\Theta(X_i)=X_i+Y^{t'}+f^{-s'}$ and $\Theta(Y_j)=Y_jf^{l''}$ for
$1\leq i\leq m$ and $1\leq j\leq n$, then the following holds:

$(a)$ $\Theta (O_P(p))$ contains a monic polynomial in $Y$ with
coefficients from $B$.

$(b)$ $\Theta (O_P(p))$ contains a polynomial $g\in B[Y]$ of the form
$1+fh$ for some $h\in B[Y]$.

Further if we choose $s'$ and $l''$ in the automorphism
$\Theta$ such that $s'>\frac{nl}{(ml-1)}l''$, then with
$e:=(ms'-nl'')ll'$, the following holds:

$(c)$ $f^{e}\Theta(a_i')\in B[Y]$ for $1\leq i\leq r-1$.

$(d)$ $f^{e}\Theta(a_r'')\in s^{2l'+2}\prod_1^n Y_i^{ll'}+fB[Y]$.

Parts $(a)$ and $(b)$ follows from (\ref{4p2}). For $(c)$, recall $l'>$  the 
maximum of total degrees of $a_1',\ldots,a_r'$, hence we only have to
ensure $e> l's'$. This is indeed the case because of our choice of $s'$.
Part $(d)$ is a direct consequence of the choice of $e$ and $s'$. 
 
Replacing $P$ by $\Theta(P)$, we assume that

$(a')$ $O_P(p)$ contains a monic polynomial in $Y$ with
coefficients from $B$.

$(b')$ $O_P(p)$ contains a polynomial $g\in B[Y]$ of the form
$1+fh$ for some $h\in B[Y]$. 

$(c')$  $f^{e}a_i'\in B[Y]$ for $1\leq i\leq r-1$.

$(d')$  $f^{e}a_r''\in s^{2l'+2}\prod_1^n Y_i^{ll'}+fB[Y]$.

We have $g=1+fh\in O_P(p)$ for some $h\in B[Y]$, hence
$A=B[Y]+gA$. Since $A=O_P(p)+s(Y-1)A$, using previous relation, we get
$A=O_P(p)+s(Y-1)B[Y]$. Therefore $B[Y]=A\cap B[Y]=O_P(p)\cap
B[Y]+s(Y-1)B[Y]$. Using $(a')$, we get $B[Y]/(B[Y]\cap O_P(p))$ is an
integral extension of $B/(B\cap O_P(p))$. Therefore, using
$B[Y]=O_P(p)\cap B[Y]+sB[Y]$, we get $B=(O_P(p)\cap B)+sB$. Hence we
get that 

$(i)$ $O_P(p)$ contains an elements $1+b's$ for some $b'\in B$.

$(ii)$ $O_P(p)$ contains an element $1+s(Y-1)a$ for some $a\in B[Y]$.

Let $\psi_1'$ and $\psi_2'$ in $P^*$ be such that $\psi'_1(p)=1+b's$
and $\psi'_2(p)=1+s(Y-1)a$. We can choose an integer $l_0>0$ such that
$f^{l_0}\,\psi_j'(p_i)\in B[Y]$ for $j=1,2$ and $1\leq i \leq r$.
Write $\phi_{r+j}=f^{l_0}\psi_j' \in P^*$ for $j=1,2$ and
$p_{r+1}=f^{e}p \in P$.

Consider the $B[Y]$-modules $M:=\sum_{i=1}^{r+1}B[Y]p_i$ and
$H:=\sum_{i=1}^{r+2}B[Y]\phi_i$. 
 We  have $\phi_i(p_j)\in B[Y]$ for $1\leq i\leq r+2$ and
$1\leq j\leq r+1$.  Further we have $M\otimes_{B[Y]} A\subset P$ and
$H\ot_{B[Y]} A\subset P^*$.

Since $sP\subset F$, we get $sp_{r+1}=\sum_{i=1}^rb_ip_i$ for some
$b_i\in A$ and hence $\phi_i(sp_{r+1})=sb_i$ for $1\leq i\leq
r$. Since $s$ is a non-zerodivisor in $A$, we get $b_i=\phi_i(p_{r+1})
\in B[Y]$. Therefore $sp_{r+1}=\sum_1^r \phi_i(p_{r+1})p_i$.  Hence if
we write $F':=\sum_{i=1}^rB[Y]p_i$, then we have $sp_{r+1}\in
F'$. Similarly if we write $G':=\sum_{i=1}^r B[Y]\phi_i$, then we get
$s\phi_{r+j}\in G'$ for $j=1,2$.  Therefore $M_s$ and $H_s$ are free
modules over $B_s[Y]$ with $M_s=F'_s$ and $H_s=G'_s$. Further $F'$ and
$G'$ are $s$-dual submodules of $M$ and $M^*$ respectively,
i.e. $sM\subset F'$, $sH\subset G'$ and the matrix
$(\phi_i(p_j))=diag(s,\ldots,s)$.
 
Let us define a $B$-algebra endomorphism $\gd:B[Y]\ra B[Y]$ by
$\gd|B=id$ and $\gd(Y)=1+(Y-1)(1-b'^2s^2)=Y+s^2b'^2(1-Y)$, where
$b'\in B$ is chosen earlier as $\phi_{r+1}(p)=f^{l_0}(1+b's)$. Since
$\gd(Y^t)-Y^t \in s^2B[Y]$ for all integers $t\geq 0$, we get that
$\gd(\ga)-\ga\in s^2B[Y]$ for any $\ga \in B[Y]$. Such an
endomorphism $\gd$ of $B[Y]$ is called $s^2$-{\it analytic} (see \cite{L95},
p 304). Recall that $M=\sum_1^{r+1}B[Y]p_i$ is a $B[Y]$-module. 

[Definition: We say
maps $\xi:M\ra M$ and $\xi^*:M^*\ra M^*$ are $\gd$-semilinear if $\xi$
and $\xi^*$ are group homomorphisms with respect to addition operation
and $\xi(\ga m)=\gd(\ga)\xi(m)$, $\xi^*(\ga\phi)=\gd(\ga)\xi^*(\phi)$ for any
$m\in M$, $\phi \in M^*$ and $\ga\in B[Y]$.]

Applying Lindel's result (\cite{L95}, Lemma 1.4) to above data, we get
$\gd$-semilinear maps $\xi:M\ra M$ and $\xi^*:M^*\ra M^*$ such that
$\xi^*(\phi)(\xi(p))=\gd(\phi(p))$ for any $\phi\in M^*$ and $p\in
M$. Further $\xi^*(H)\subset H$. Therefore $A\otimes_{B[Y]}
\xi^*(H)\subset P^*$. 

We have the followings:

$(i')$ $\phi_r(p_{r+1})=\phi_r(f^{e}p)=s^{2l'+2}\prod_1^n Y_i^{ll'}+f\wt b$ 
for some $\wt b\in B[Y]$, using $(d')$. 

$(ii')$
$\phi_{r+1}(p_{r+1})=f^{l_0}\psi'_1(f^{e}p)=f^{l_0+e}(1+b's)$ for
$b'\in B$, using $(i)$.

$(iii')$
$\phi_{r+2}(p_{r+1})=f^{l_0}\psi'_2(f^{e}p)=f^{l_0+e}(1+s(Y-1)a)$
for some $a\in B[Y]$, using $(ii)$.

Using the relation $\xi^*(\phi)(\xi(p_{r+1}))=\gd(\phi(p_{r+1}))$, we
see that the $\gd$-images of elements in $(i')-(iii')$ belong to
$O_M(\xi(p_{r+1}))$. Further using $A\otimes_{B[Y]} \xi(M)\subset P$
and $A\otimes_{B[Y]} \xi^*(H)\subset P^*$, we see that $\gd$-images of
above three elements belong to $O_P(1\ot \xi(p_{r+1}))$. We will show
that $1\otimes \xi(p_{r+1})$ is a unimodular element of $P$ by showing
that the $\gd$-images of above three elements generate the unit ideal.
Suppose not, then there exist a
maximal ideal $\mm$ containing elements 

$(i'')$ $\gd(s^{2l'+2}\prod_1^n Y_i^{ll'}+f\wt b)=s^{2l'+2}\prod_1^n Y_i^{ll'} + 
\gd(f)\gd(\wt b)$, 

$(ii'')$ $\gd(f^{l_0+e}(1+b's))= \gd(f)^{l_0+e}(1+b's)$ and

$(iii'')$
$\gd(f^{l_0+e}(1+s(Y-1)a))=\gd(f)^{l_0+e}(1+s\gd(Y-1)\gd(a))$.

Assume $\gd(f)\in \mm$. Then by $(i'')$, we get that $s\in \mm$ as
$Y_i$'s are units in $A$. Since $\gd(f)-f\in (s^2)$, we get $f\in
\mm$. This is a contradiction, since $f$ is a unit in $A$.  In the other
case, assume $\gd(f)\notin \mm$.  Then $1+s\gd(Y-1)\gd(a)\in
\mm$  and $1+b's\in \mm$. Since $\gd(Y-1)=(Y-1)(1-b'^2s^2)\in
(1+b's)A$, we get $\gd(Y-1)\in \mm$. This shows that $1\in \mm$,
a contradiction. Therefore we get that $1\otimes \xi(p_{r+1})$ is a
unimodular element. This completes the proof.  $\hfill \gj$

\subsection{Proof of Theorem \ref{T2}}

Without loss of generality, we assume $n\geq 1$. If $C:=
R[X_1,\ldots,X_m,Y_1^{\pm 1},\ldots,Y_{n-1}^{\pm 1}] $, then
$A=C[Y_n^{\pm 1},Y,f^{-1}]$ with $f\in R[Y]$. We are given a
surjection $\phi: P\ot (A/I) \surj I/I^2$, where $P=Q\op R$. We
want to show that $\phi$ can be lifted to a surjection $\Phi : P\ot A
\surj I$.

Let $\Phi_1:P\ot A\ra I$ be a lift of $\phi$.  We can find an integer
$k>0$ such that $f^{k}\Phi_1$ maps $P\ot C[Y_n^{\pm 1},Y]$ into
$J:=I\cap C[Y_n^{\pm 1},Y]$. Now $f^k \Phi_1$ induces a map $\psi
:P\ot (C[Y_n^{\pm 1},Y]/J) \ra J/J^2$.  Note that
$\psi_f=f^k\phi :P\ot (A/I) \surj
(J/J^2)_f$ is a surjection.

Note that in the proof of (\ref{4p2}, part $(i)$), $f$ being monic is
not needed. Using height of $I>d$ and applying (\ref{4p2}), we get an
$R[Y,f^{-1}]$-automorphism $\Theta$ of $A$ such that $\Theta(I)$
contains $1+fh$ for some $h\in C[Y_n^{\pm 1},Y]$.  Replacing $A$ by
$\Theta(A)$ and $I$ by $\Theta(I)$, we can assume that $1+fh\in
I$. Since $1+fh \in J$, we get $(J/J^2)_{1+fh}$ is the zero
module. Hence $\psi_{1+fh}$ is a surjections.  Applying (\ref{4p3}),
we get $\psi : P\ot (C[Y_n^{\pm 1},Y]/J) \ra J/J^2$ is a
surjection. If $\psi$ has a surjective lift $\Psi: P\ot C[Y_n^{\pm
    1},Y] \surj J$, then $f^{-k}\Psi_f :P\ot A \surj I$ will be our
required surjective lift of $\phi$. Therefore it is enough to show
that $\psi$ has a surjective lift from $P\ot C[Y_n^{\pm 1},Y]$ onto
$J$.

Note that $C[Y_n^{\pm 1},Y]=B[Y]$ is a Laurent polynomial ring over $R$ and
$J$ is an ideal of $C[Y_n^{\pm 1},Y]$ of height $>d=\dim R$.  By
(\ref{4l2}), there exist a $R[Y_n^{\pm 1}]$-automorphism $\Theta$ of
$C[Y_n^{\pm 1},Y]$ such that $\Theta(J)$ contains a monic polynomial
in $Y_n$ of the form $1+Y_nh'$ for some $h'\in C[Y,Y_n]$. Replacing
$J$ by $\Theta(J)$, we can assume that $J$ contains a monic polynomial
$1+Y_nh'$ in the variable $Y_n$.

Lift $\psi$ to a map $\Psi_1:P\ot C[Y,Y_n^{\pm 1}] \ra J$.  If we write
$K:=J\cap C[Y,Y_n]$, then $Y_n^{l}\Psi_1$ will map $P\ot C[Y,Y_n]$
into $K$ for some integer $l>0$. Now $Y_n^l \Psi_1$ will induce a map
$\gd: P\ot (C[Y,Y_n]/K) \ra K/K^2$ such that $\gd_{Y_n}=Y_n^l\psi$ is a
surjection from $P\ot (C[Y,Y^{\pm 1}_n]/J)$ onto $J/J^2$.
Since $K$ contains a monic polynomial $1+Y_nh'$ in $Y_n$, we get
$(K/K^2)_{1+Y_nh'}=0$. Applying (\ref{4p3}), we get that $\gd : P\ot
(C[Y,Y_n]/K) \surj K/K^2$ is a surjection. Applying Bhatwadekar-Raja
Sridharan (\ref{4p4}), we get that $\gd$ can be lifted to a surjection
$\Delta: P\ot C[Y,Y_n] \surj K$. Therefore $Y_n^{-l}\gD$ is a
surjective lift of $\psi$.  This completes the proof.  $\hfill \gj$

\subsection{Proof of Theorem \ref{T3}}

Without loss of generality we assume that $f\in R[Y]-R$. Let
$\Phi_1:P[Y,f^{-1}]\ra I$ be any lift of $\phi$. Then $\Phi_1(1)=\gd$
modulo $I(1)^2$.  Hence $\Phi_1(1)-\gd\in I(1)^2Hom(P,B)$.  Write
$\Phi_1(1)-\gd=f_1(1)g_1(1)\ga_1+\cdots+f_r(1)g_r(1)\ga_r$ for some
$f_i,g_i\in I$ and $\ga_i \in Hom(P,B)$.  If we write
$\Phi_2:=\Phi_1-(f_1g_1\wt\ga_1+\cdots +f_rg_r\wt\ga_r)$, where $\wt
\ga_i=\ga_i\ot id :P\ot_B A \ra A$, then $\Phi_2:P[Y,f^{-1}] \ra I$ is
also a lift of $\phi$ with $\Phi_2(1)=\gd$.
 
Let $J:=I\cap B[Y]$. Then there exist $k>0$ such that $f^k\Phi_2$ maps
$P[Y]$ into $J$. Now $f^k\Phi_2$ induces a map $\psi: P[Y]/JP[Y] \ra
J/J^2$.  Note that $\psi_f=f^k\phi :P[Y,f^{-1}]/IP[Y,f^{-1}] \surj
(J/J^2)_f$ is a surjection.

Since $\hh I>d$, by (\ref{4p2}), applying some
$R[Y,f^{-1}]$-automorphism of $A$, we may assume that $I$ contains
$(i)$ a monic polynomial $g$ in $Y$ with coefficients from $B$ and
$(ii)$ an element $1+fh$ for some $h\in B[Y]$. Since $1+fh\in J$, we
get $(J/J^2)_{1+fh}=0$. Therefore $\psi_{1+fh}$ is the zero map. By
(\ref{4p3}), we get $\psi : P[Y]/JP[Y] \surj J/J^2$ is a surjection.
Further $f(1)^k\gd :P \surj J(1)$ is a surjection with $\psi(1) =
f(1)^k \gd \ot B/J(1)$. Since rank of $P\geq \dim B[Y]/J+2$ holds and
$J$ contains monic polynomial $g$, using Mandal (\cite{Man}, Theorem
2.1), there exist a surjection $\Psi: P[Y]\surj J$ which is a lift of
$\psi$ and $\Psi(1)=f(1)^k\gd$. Therefore
$\Phi=f^{-k}\Psi_f:P[Y,f^{-1}] \surj I$ is a surjection which is a
lift of $f^{-k}\psi=\phi$ with $\Phi(1)=\gd$.  This completes the
proof.  $\hfill \gj$


\section{Applications}

Let $M$ be a finitely generated
$R$-module. If we write $\mu(M)$ for the minimum number of generators of
$M$ as an $R$-module, then F\"oster \cite{F67} and Swan \cite{Swan67}
proved that $\mu(M)\leq max\{\mu(M_{\pri})+\dim (R/\pri) | \pri \in
\Spec (R), M_{\pri}\neq 0\}$.  In particular, if $P$ is a projective
$R$-module of rank $r$, then $\mu(P)\leq r+d$. As a consequence of our
result (\ref{T1}), we prove the following result.

\begin{theorem}
Let $A=B[Y,f^{-1}]$ for some monic polynomial $f\in R[Y]$ 
and $P$ a projective $A$-module of rank $r$. Then
$\mu(P)\leq r+d$.
\end{theorem}

\begin{proof}
Assume $P$ is generated by $s$ elements, where $s>r+d$. Then we will
show that $P$ is also generated by $s-1$ elements.  By F\"orster-Swan,
 we have $s\leq \dim A +r=d+m+n+1+r$. Let $\phi:A^s \surj P$ be a
surjection. If $Q$ is the kernel of $\phi$, then rank of $Q$ is
$s-r>d$. Hence by (\ref{T1}), $Q$ has a unimodular element, say $q\in
\Um(Q)$. Since $A^s\iso P\op Q$, we get $q\in \Um(A^s)$. Since
$\phi(q)=0$, $\phi$ induces a surjection $\ol \phi : A^s/qA \surj P$.
Since $s-1>d$, by (\ref{DK2}), $A^{s-1}$ is cancellative. Hence $A^s/qA\iso
A^{s-1}$. Therefore $P$ is generated by $s-1$ elements. This completes
the proof.
$\hfill \gj$
\end{proof}

\begin{proposition}\label{4t5}
  Let $A=B[Y,f^{-1}]$ for some $f\in R[Y]$.  Let $J\subset I$ be two
  ideals of $A$ such that $I=(f_1,\ldots,f_n)+ I^2$ and
  $J=(f_1,\ldots,f_{n-1}) + I^{(n-1)!}$. Assume that $I$ contains
  $(i)$ a monic polynomial $F\in C[Y]$ in the variable $Y$ and $(ii)$
  an element of the form $1+fh$ for some $h\in C[Y]$. Then $J$ is
  generated by $n$ elements. As a consequence, $I$ is
  set-theoretically generated by $n$ elements.
\end{proposition}

\begin{proof}
 Replacing $f_i$ by $f^Nf_i$ for some integer $N>0$, we may assume
 that $f_i\in B[Y]$ for all $i$.  Let $K=I\cap B[Y]$ be an ideal of
 $B[Y]$. Let $\phi:(B[Y]/K)^n \ra K/K^2$ be the map defined by
 $e_i\mapsto \ol f_i$. Then $\phi_f$ is surjective and $\phi_{1+fh}$
 is zero map, since $1+fh\in K$. Hence by (\ref{4p3}), $\phi$ is a
 surjection. Therefore, we get $K=(f_1,\ldots,f_n)+ K^2$.  If
 $L:=(f_1,\ldots,f_{n-1})+ K^{(n-1)!}$, then $L_f=J$. Since $K$
 contains a monic polynomial $F$, using (\ref{4t4}), we get that $L$
 is generated by $n$ elements.
Therefore $J$ is generated by $n$ elements.  
$\hfill \gj$
\end{proof}

\begin{theorem}
 Let $A=B[Y,f^{-1}]$ for some $f\in R[Y]$.  Let $J\subset I$ be two
 ideals of $A$ such that $I=(f_1,\ldots,f_n)+ I^2$ and
 $J=(f_1,\ldots,f_{n-1}) + I^{(n-1)!}$. Assume that height of $I>d$.
 Then $J$ is generated by $n$ elements. In particular, $I$ is
 set-theoretically generated by $n$ elements.
\end{theorem}

\begin{proof}
Note that in the proof of (\ref{4p2}) part $(i)$, $f$ being monic is
not used. Hence applying an automorphism as in (\ref{4p2}), we may
assume that $I$ contains an element $1+fh$ for some $h\in B[Y]$.  Now
as in (\ref{4t5}), replacing $f_i$ by $f^Nf_i$, we may assume that
$f_i\in B[Y]$. Then if $K=I\cap B[Y]$, then $K=(f_1,\ldots,f_n)+K^2$
as in (\ref{4t5}). Since height of $K>d$, using some automorphism of
$B[Y]$, we may assume that $K$ contains a monic polynomial in
$Y$. Now, if $L=(f_1,\ldots,f_{n-1})+K^{(n-1)!}$, then by (\ref{4t4}),
$L$ is generated by $n$ elements. Hence $J=K_f$ is generated by $n$
elements.  $\hfill \gj$
\end{proof}

\begin{theorem}
  Let $A=B[Y,f^{-1}]$ for some $f\in R[Y]$ with further condition that
  $m+n\geq 1$. Let $I\subset A$ be a locally complete intersection
  ideal of height $r \geq max\{\dim A-1,\dim A-r+2\}$ with $\dim A/I
  \leq 1$. Then $I$ is set theoretically generated by $r$ elements.
\end{theorem}

\begin{proof}
By Ferrand-Szpiro (\ref{4t6}), there is a locally complete
intersection ideal $J$ of height $r$ such that $(i)$
$\sqrt{J}=\sqrt{I}$ and $(ii)$ $J/J^2$ is a free $A/J$-module of rank
$r$. Since $m+n\geq 1$, we get $r\geq d+1$.  By
(\ref{T2}), the $r$ generators of free module $J/J^2$ can be lifted to
$r$ generators of $J$. Hence $I$ is set theoretically generated by $r$
elements.  $\hfill \gj$
\end{proof}

{\bf Acknowledgment:} The second author would like to thank C.S.I.R.
India for their fellowship.

{\small
{}

}

\end{document}